\documentclass[reqno]{amsart}
\usepackage{amsmath,mathtools,amssymb}

\mathtoolsset{showonlyrefs}

\usepackage{amsmath}
\usepackage{amsthm}
\usepackage{amsfonts}
\usepackage{amssymb}
\usepackage{mathrsfs}
\usepackage[shortlabels]{enumitem}
\usepackage{color}
\usepackage{hyperref}
\usepackage[normalem]{ulem}

\newtheorem{theorem}{Theorem}
\newtheorem{proposition}[theorem]{Proposition}
\newtheorem{lemma}[theorem]{Lemma}

\theoremstyle{definition}
\newtheorem{definition}{Definition}
\newtheorem{remark}{Remark}

\newcommand{\N}{\mathbb{N}}
\newcommand{\Z}{\mathbb{Z}}

\newcommand{\R}{\mathbb{R}}
\newcommand{\C}{\mathbb{C}}

\newcommand\e{\mathrm{e}}
\newcommand\I{\mathrm{i}}

\newcommand\re{\operatorname{Re}}
\newcommand\im{\operatorname{Im}}

\newcommand\eps\varepsilon
\renewcommand\epsilon\varepsilon
\renewcommand\rho\varrho
\newcommand\lm\lambda

\newcommand{\supp}{\operatorname{supp}}
\newcommand{\dist}{\operatorname{dist}}

\newcommand{\rd}{\mathrm{d}}

\title[Eigenvalue sum bounds for radial potentials]{BOUNDS FOR EIGENVALUE SUMS OF SCHRODINGER OPERATORS 
WITH COMPLEX RADIAL POTENTIALS}
 \author{Jean-Claude Cuenin}
 \address{Department of Mathematical Sciences, Loughborough University, Loughborough,
 Leicestershire, LE11 3TU United Kingdom}
 \email{J.Cuenin@lboro.ac.uk}
 \author{Solomon Keedle-Isack}
  \address{Department of Mathematical Sciences, Loughborough University, Loughborough,
 Leicestershire, LE11 3TU United Kingdom}
 \email{S.Keedle-Isack@lboro.ac.uk}
\date{\today}
\begin{document}

\begin{abstract}
We consider eigenvalue sums of Schr\"odinger operators $-\Delta+V$ on $L^2(\R^d)$ with complex radial potentials $V\in L^q(\R^d)$, $q<d$. We prove quantitative bounds on the distribution of the eigenvalues in terms of the $L^q$ norm of $V$. 
A consequence of our bounds is that, if the eigenvalues $(z_j)$ accumulate to a point in $(0,\infty)$, then $(\im z_j)$ is summable. The key technical tools are resolvent estimates in Schatten spaces. We show that these resolvent estimates follow from spectral measure estimates by an epsilon removal argument.
\end{abstract}
\maketitle

\section{Introduction and main results}\label{section: introduction}
We consider Schrödinger operators $H=-\Delta+V$ on $L^2(\R^d)$ with complex-valued potentials $V\in L^q(\R^d)$, with $q<\infty$.
The spectrum of $H$ consists of $[0,\infty)$ together with a discrete set of eigenvalues $z_j$.
We are interested in quantitative bounds on the $z_j$ that depend only on an $L^q$ norm of the potential. 

\subsection{Bounds for single eigenvalues}
Laptev and Safronov \cite{MR2540070} conjectured that, in $d\geq 2$ dimensions, any non-positive eigenvalue $z$ of $H$ satisfies the bound
\begin{align}\label{LT bound}
|z|^{\gamma}\leq D_{\gamma,d}\int_{\R^d}|V(x)|^{\gamma+\frac{d}{2}}\rd x
\end{align}      
for $0<\gamma\leq d/2$, and with $D_{\gamma,d}$ independent of $V$ and $z$. Prior to this, Abramov-Aslanyan-Davies \cite{MR1819914} had shown that \eqref{LT bound} holds if $d=1$ and $\gamma=1/2$. For $d\geq 2$, the Laptev--Safronov conjecture was proved for $0<\gamma\leq 1/2$ by Frank \cite{MR2820160} and disproved for $\gamma>1/2$ by B\"ogli and the first author \cite{MR4561804}. The conjecture is true for \emph{radial} potentials in the range $0<\gamma<d/2$, as proved by Frank--Simon \cite{MR3713021}, and it fails for $\gamma\geq d/2$ by a counterexample of B\"ogli \cite{MR3627408}. There are some further refinements and generalisations of \eqref{LT bound} (see, e.g., \cite{MR4104544} and references therein), and there is by now a more or less complete picture of bounds of the type \eqref{LT bound} for a single eigenvalue $z$.

\subsection{Bounds for eigenvalue sums}
The situation for eigenvalue sums is considerably less well understood. The starting point is the celebrated Lieb--Thirring inequality for \emph{real-valued potentials},
  \begin{align}\label{LT bound sums}
\sum_j |z_j|^{\gamma}\leq L_{\gamma,d} \int_{\R^d}|V|^{\gamma+d/2}\rd x,
\end{align}  
which holds for $\gamma\geq 1/2$ if $d=1$ and $\gamma>0$ if $d\geq 2$. Frank--Laptev--Lieb--Seiringer \cite{MR2260376} proved that, for $\gamma\geq 1$, the inequality \eqref{LT bound sums} holds for all eigenvalues outside a fixed cone, $|\im z_j|\geq \kappa \re z_j$, with a constant that blows up as $\kappa\to 0$ at a rate $\kappa^{-\gamma-d/2}$. In $d=1$, the blowup rate was shown to be optimal by B\"ogli \cite{MR3627408}. Averaging the bound of Frank--Laptev--Lieb--Seiringer with respect to $\kappa$, Demuth--Hansmann-Katriel \cite{MR2559715} proved that
\begin{align}\label{DKH bound}
 \sum_j |z_j|^{\gamma}\left(\frac{\delta(z_j)}{|z_j|}\right)^{\gamma+d/2+\eps}\leq C_{\gamma,d,\eps} \int_{\R^d}|V|^{\gamma+d/2}\rd x,
\end{align}
where $\delta(z):=\dist(z,[0,\infty))$, $\gamma\geq 1$ and $\eps>0$. They also posed the question \cite{MR3016473} whether \eqref{DKH bound} is true for $\eps=0$. For $d=1$, B\"ogli and Štampach \cite{MR4322041} answered this question in the negative.

A rather different set of results concerning eigenvalue sums have been established by Frank--Sabin \cite{MR3730931} and Frank \cite{MR3717979}. These bounds are of the (scale-invariant) form
\begin{align}\label{Frank--Sabin type bounds}
    \left(\sum_j |z_j|^{\alpha}\left(\frac{\delta(z_j)}{|z_j|}\right)^{\beta}\right)^{\gamma/\alpha}\leq C_{\alpha,\beta,\gamma,d} \int_{\R^d}|V|^{\gamma+d/2}\rd x.
\end{align}
It would be too technical for this introduction to state the precise values of $\alpha,\beta,\gamma,d$ for which \eqref{Frank--Sabin type bounds} holds, so we will only remark some general features of these bounds.\footnote{We are indebted to Rupert Frank for communicating these remarks very clearly during an online IAMP seminar \cite{FrankIAMP}.} 
\begin{itemize}
    \item If $0<\gamma\leq 1/2$ ($\gamma=1/2$ if $d=1$) then $\beta=1$. This means that if the eigenvalues accumulate in $(0,\infty)$ then they are summable, i.e., $(\im z_j)\in\ell^1$.\\
    \item  In all known bounds, we have $\gamma/\alpha<1$. In other words, a sum is bounded by a \emph{power} of an integral strictly greater than $1$. This means that there is a loss of \emph{locality}, which is a crucial feature of the Lieb--Thirring inequality \eqref{LT bound sums}.\\
    \item One particular bound from \cite{MR3717979} for $d=1$ states that
\begin{align}\label{Frank III d=1}
\left(\sum_j \delta(z_j)^{\beta}\right)^{1/(2\beta)}\leq C_{\beta} \int_{\R^d}|V(x)|\rd x
\end{align}
holds with $\beta=1$. The first author showed in \cite{MR4426735} that the inequality fails for $\beta<1$, so the bound \eqref{Frank III d=1} is optimal.
\end{itemize}
The example of \cite{MR4426735} also shows that one can have unexpectedly many eigenvalues compared to the number of resonances. In higher dimensions, it is an open problem whether there are Schrödinger operators with complex potentials
that have significantly more eigenvalues than their real counterparts. 

\subsection{Main results}
Our main contribution is to improve the results of Frank \cite{MR3717979}, concerning bounds of the type \eqref{Frank--Sabin type bounds}, under the assumption that $V$ is \emph{radial}. As we already mentioned, radial potentials satisfy better estimates in the case of individual eigenvalues. To the best of our knowledge, such an effect has not been observed so far for sums of eigenvalues.

\begin{theorem}\label{theorem sum of eigenvalues upper bound}
Let $d\geq2$, $q\in(\frac{d+1}{2},d)$, $p>\frac{(d-1)q}{d-q}$. Then for radial $V\in L^{q}(\mathbb{R}^{d})$, the eigenvalues $z_{j}$ of $H$ satisfy 
\begin{align}\label{eq. sums of eigenvalues main theorem}
\left(\sum_{j}\delta(z_j)|z_{j}|^{p(1-\frac{d}{2q})-1}\right)^{q/p}\leq C_{p,q}\int_{\R^d}|V|^q\rd x.
\end{align}
In particular, if $(z_j)$ accumulates to a point in $(0,\infty)$, then $(\im z_j)\in \ell^1$.
\end{theorem}

\begin{remark}
  i)  Frank-Sabin \cite{MR3730931} proved the case $q\in(\frac{d}{2},\frac{d+1}{2}]$ with the same accumulation rate. 
  \\\\
  ii) Frank \cite{MR3717979} proved similar bounds for $q>\frac{d+1}{2}$, but with a slower accumulation rate. Moreover, Frank's bounds distinguish between eigenvalues lying in a disk around the origin and eigenvalues lying outside this disk.\\\\
  iii) The overall powers of $|z_{j}|$ in the bounds of Frank--Sabin \cite{MR3730931} and Frank \cite{MR3717979} are always negative, whereas ours are positive.\\\\
  iv) Since $q/p<(d-q)/(d-1)<1$, the bound \eqref{eq. sums of eigenvalues main theorem} exhibits the same non-locality as those in \cite{MR3730931}, \cite{MR3717979}. We conjecture that the exponent $q/p$ in \eqref{eq. sums of eigenvalues main theorem} cannot be increased, i.e.\ that the inequality
  \begin{align}
      \left(\sum_{j}\delta(z_j)^{\beta}|z_{j}|^{\beta p(1-\frac{d}{2q})-1}\right)^{q/(\beta p)}\leq C_{p,q,\beta}\int_{\R^d}|V|^q\rd x
  \end{align}
  fails for $\beta<1$. We leave it as an open problem to find a counterexample.\\\\
  v) Another interesting question is whether one can dispense with the radiality assumption and instead replace the right-hand side of \eqref{eq. sums of eigenvalues main theorem} by a constant multiple of the mixed norm
  \begin{align}
     \int_{0}^{\infty} \sup_{\omega \in S^{d-1} }|V(r\omega)|^q r^{d-1}\rd r. 
  \end{align}
  Such generalizations for bounds of single eigenvalues appear in \cite{MR3713021}.
\end{remark}

The proof of Theorem \ref{theorem sum of eigenvalues upper bound} follows the same general strategy as in \cite{MR3730931}, \cite{MR3717979}. It is based on the identification of eigenvalues of the Schr\"odinger operator $H$ with zeroes of an analytic function (a regularized determinant). This method was
pioneered by Demuth--Katriel \cite{MR2413204} and Borichev--Golinskii--Kupin \cite{MR2481997} and extended by Demuth--Hansmann--Katriel \cite{MR2559715,MR3077277}. The method rests upon a remarkable generalization of Jensen's inequality for analytic functions, due to \cite{MR2481997}, for functions that blow up at some points of the boundary. We will appeal to the quantitative version of Frank \cite[Theorem 3.1]{MR3717979}, but the special case $\rho=0$ there is essentially contained in the proof of \cite[Theorem 16]{MR3730931}. It corresponds to a blow-up at a single point ($z=0$). The main new technical ingredient in our proof is a uniform resolvent bound in trace ideals. 

\begin{theorem}\label{theorem resolvent estimate}
Let $d/2\leq q<d$ and $p>\frac{(d-1)q}{d-q}$.
Then for all $z\in\C\setminus[0,\infty)$ and for all radial functions $W_1,W_2\in L^{q}(\R^d)$,     
\begin{align}\label{eq. resolvent estimate}
        \|W_1(-\Delta-z)^{-1}W_2\|_{\mathfrak{S}^p(L^2(\R^d))}\leq C_{p,q,d} |z|^{\frac{d}{2q}-1}\|W_1\|_{L^{2q}}\|W_2\|_{L^{2q}}.
    \end{align}  
\end{theorem}

\begin{remark}
    i) Theorem \ref{theorem resolvent estimate} extends the uniform resolvent estimates of Frank--Sabin \cite[Theorem 12]{MR3730931} from $q\leq (d+1)/2$ to $q<d$ under the assumption of radial symmetry.\\\\
    ii) Our Schatten exponent $p$ coincides with that of \cite{MR3730931} and is optimal in the range $q\leq (d+1)/2$, as shown in \cite[Theorem 6]{MR3730931}. The resolvent bounds of Frank--Sabin hold for arbitrary potentials, not just radial potentials. However, their counterexample is radial and the argument remains valid for $q>(d+1)/2$, so the same example shows optimality of the Schatten exponent $p$ in Theorem \ref{theorem resolvent estimate}.\\\\
    iii) The Schatten norm bound in Theorem \ref{theorem resolvent estimate} is crucial for applications involving sums of eigenvalues. For bounds involving only individual eigenvalues, a weaker bound for the operator norm would suffice. By H\"older's inequality, the operator norm bound is equivalent to an $L^p\to L^{p'}$ bound for $(-\Delta-z)^{-1}$. In the range $d/2\leq q\leq (d+1)/2$, this is a special case of the uniform resolvent bounds of Kenig--Ruiz--Sogge \cite{MR894584}. 
    The idea of using uniform resolvent bounds in the context of eigenvalue bounds for Schr\"odinger operators with complex potentials is due to Frank \cite{MR2820160}.\\\\
    iv) For radial potentials, Frank--Simon proved uniform $L^p\to L^{p'}$ bounds (equivalently, bounds with the operator norm instead of a Schatten norm in the left hand side of \eqref{eq. resolvent estimate}) in the optimal range $d/2\leq q<d$. Thus, one of our main contributions is to upgrade their bounds to stronger trace ideal bounds. 
\end{remark}

\section{Spectral Measure Estimate}
Let $E(\lambda)=\mathbf{1}_{[0,\lambda]}(\sqrt{-\Delta})$. The spectral function can be factorized as
\begin{align}\label{spectral measure factorization}
   \rd E(\lambda)/\rd\lambda=c_d\lambda^{d-2}\mathcal{E}(\lambda)\mathcal{E}(\lambda)^*,
\end{align}
 where $\mathcal{E}(\lambda):L^1(\mathbb{S}^{d-1})\to L^{\infty}(\R^d)$ is the Fourier extension operator
\begin{align}
\mathcal{E}(\lambda)g(x)=\int_{\mathbb{S}^{d-1}}\e^{\I\lambda x\cdot\xi}
g(\xi)\rd S(\xi),\quad x\in\R^d,\quad \lambda>0,
\end{align}
and where $\rd S$ denotes induced Lebesgue measure on the unit sphere $\mathbb{S}^{d-1}$. If $\lambda=1$, we will write $\mathcal{E}:=\mathcal{E}(1)$.

The main result of this section is the following precursor to Theorem \ref{theorem resolvent estimate}.

\begin{theorem}\label{theorem spectral measure bound}
Let $1\leq q<d$ and $p>\frac{(d-1)q}{d-q}$.
Then for all $\lambda>0$ and for all $W_1,W_2\in L^{2q}_{\rm rad}(\R^d)$,     
\begin{align}\label{eq. spectral measure bound}
        \|W_1(\rd E(\lambda)/\rd\lambda)W_2\|_{\mathfrak{S}^p(L^2(\R^d))}\leq C_{p,q,d} \lambda^{\frac{d}{q}-2}\|W_1\|_{L^{2q}}\|W_2\|_{L^{2q}}.
    \end{align}  
\end{theorem}

\begin{proof}
    By scaling, we may assume without loss of generality that $\lambda=1$.
    By \eqref{spectral measure factorization} and a $TT^*$ argument (to reduce to the case $W_1=\overline{W_2}$), it then suffices to prove 
\begin{align}
\|W\mathcal{E}\mathcal{E}^*\overline{W}\|_{\mathfrak{S}^p(L^2(\R^d))}\leq C_{p,q,d} \|W\|_{L^{2q}}^2,\quad\forall W\in L_{\rm rad}^{2q}(\R^d),
    \end{align}  
or equivalently
\begin{align}\label{eq. EW^2E bound}
\|\mathcal{E}V\mathcal{E}^*\|_{\mathfrak{S}^p(L^2(\R^d))}\leq C_{p,q,d} \|V\|_{L^{q}}\quad\forall V\in L_{\rm rad}^q(\R^d).
    \end{align}  
To prove \eqref{eq. EW^2E bound}, we set
\begin{align*}
\Sigma:=\mathcal{E}^*V\mathcal{E}:L^2(\mathbb{S}^{d-1})\to L^2(\mathbb{S}^{d-1}).
\end{align*}
We use $\mathcal{H}_{k}$ and $\mathcal{A}_{k}$ to denote the $k$-th degree spherical harmonics and solid spherical harmonics respectively. Since $L^{2}(\mathbb{S}^{d-1})=\bigoplus_{k=0}^{\infty}\mathcal{H}_{k}$, we can decompose $\Sigma$ into an orthogonal sum $\Sigma=\bigoplus_{k=0}^{\infty}\Sigma_{k}$, where
\begin{align}
\Sigma_{k}:=\mathcal{E}^{*}V\mathcal{E}|_{\mathcal{H}_{k}}:\mathcal{H}_{k}\to\mathcal{H}_{k}.
\end{align}
The fact that $\Sigma_k$ maps $\mathcal{H}_{k}$ to itself follows from corresponding mapping properties of the Fourier transformation. More precisely, it follows from \cite[Thm. 3.10]{MR304972} that if $f\in (L^1\cap L^2)(\R^d)$ and $f(x)=f_0(|x|)P(x/|x|)$, where $P\in\mathcal{A}_k$, then $\widehat{f}(\xi)=F_0(|\xi|)P(\xi)$, where
\begin{align*}
F_0(r)=2\pi \I^{-k}r^{-(d+2k-2)/2}\int_0^{\infty}f_0(s)J_{(d+2k-2)/2}(2\pi rs)s^{(d+2k)/2}\rd s,
\end{align*}
and $J_{\nu}$ is a Bessel function.
Note that $P$ is homogeneous of degree $k$, i.e.\ $P(\xi)=r^kP(\omega)$, $r=|\xi|$, $\omega=\xi/r$. We set $H=P|_{\mathbb{S}^{d-1}}\in\mathcal{H}_k$ and $\widetilde{f_0}(r):=r^kf_0(r)$. Then we have
$\widehat{f}(\xi)=\widetilde{F_0}(r)H(\omega)$, with 
\begin{align*}
\widetilde{F_0}(r)=2\pi \I^{-k}r^{-(d-2)/2}\int_0^{\infty}\widetilde{f_0}(s)J_{(d+2k-2)/2}(2\pi rs)s^{d/2}\rd s.
\end{align*}
Hence, the Fourier transform of the function $x\mapsto\widetilde{f_0}(|x|)H(x/|x|)$, where $H\in\mathcal{H}_k$, is the function $\xi\mapsto\widetilde{F_0}(r)H(\omega)$.
In particular,
\begin{align*}
\mathcal{E}^*(x\mapsto\widetilde{f_0}(|x|)H(x/|x|))(r\omega)=\widetilde{F_0}(1)H(\omega)
=\langle \varphi_k,\widetilde{f_0}\rangle_{L^2(\R_+,r^{d-1}\rd r)}H(\omega),
\end{align*}
where 
\begin{align*}
\varphi_k(s):=2\pi \I^{-k}J_{(d+2k-2)/2}(2\pi s)s^{-d/2+1}.
\end{align*}
A straightforward calculation shows that, for $H\in \mathcal{H}_k$,
\begin{align*}
\Sigma_k H=\lambda_k H,\quad \lambda_k:=\left(\int_0^{\infty}|\varphi_k(s)|^2v(s)s^{d-1}\rd s\right),\quad v(|x|):=V(x).
\end{align*}
Hence, $\Sigma_k$ is a multiple of the identity $\mathbf{1}_{\mathcal{H}_k}$, with eigenvalue $\lambda_k$ of multiplicity $\dim\mathcal{H}_k\approx k^{d-2}$. By Hölder,
\begin{align*}
\lambda_k&=(2\pi)^2\int_0^{\infty}|J_{(d+2k-2)/2}(2\pi s)|^2w(s)s\rd s\\
&\leq (2\pi)^2 \left(\int_0^{\infty}|J_{(d+2k-2)/2}(2\pi s)|^{2p}s^{d-1+q'(2-d)}\rd s\right)^{1/q'}\left(\int_0^{\infty}|v(s)|^{q}s^{d-1}\rd s\right)^{1/q}.
\end{align*}
If we set $\varrho=d-1+q'(2-d)$, then the conditions of Lemma~\ref{Frank-Simon Bessel integral estimate} below are satisfied for $d\geq2$ and 
\begin{align}
    q'>\max\left\{\frac{d}{d-1},\frac{d-\frac{2}{3}}{d-\frac{4}{3}}\right\}=\frac{d}{d-1}\iff q<d.
\end{align}
Hence, for $q<d$, we obtain 
\begin{equation}\nonumber
    \lambda_{k}\lesssim\max\left\{k^{-q'+\varrho+1},k^{-\frac{2q'}{3}+\varrho+\frac{1}{3}}\right\}^{1/q'}\|V\|_{L^q}.
\end{equation}
By orthogonality, we then have 
\begin{align*}
\|\Sigma\|_{\mathfrak{S}^{p}}^{p}=\sum_{k=0}^{\infty}\|\Sigma_k\|_{\mathfrak{S}^{p}}^{p}=\sum_{k=0}^{\infty}\dim\mathcal{H}_k\lambda_k^{p}\lesssim \sum_{k=0}^{\infty}k^{d-2}\lambda_k^{p}\lesssim \|V\|_{L^q}^{p},
\end{align*}
provided that $q<d$ and 
\begin{align*}
d-2+\max(-q'+\rho+1,-\frac{2}{3}q'+\rho+\frac{1}{3})\frac{p}{q'}<-1\iff p>\frac{(d-1)q}{d-q}.
\end{align*}
This completes the proof of Theorem \ref{theorem spectral measure bound}.
\end{proof}

In the previous proof we used
the following lemma concerning bounds on integrals of Bessel functions. It is a special case of a result of Barcelo et al. \cite{BarceloRuizVega}. The specific bound we require is stated below and can be inferred from \cite[Lemma A.3]{MR3713021}. 

\begin{lemma}\label{Frank-Simon Bessel integral estimate}
    For $\varrho\in\mathbb{R}$ and $p>\varrho+1$, $\frac{2p}{3}>\varrho+\frac{1}{3}$, we have 
    \begin{equation}\nonumber
        \int_{1}^{\infty}|J_{\mu}(2\pi s)|^{2p}s^{\varrho}\,ds\lesssim\max\left\{\mu^{-p+\varrho+1},\mu^{-\frac{2p}{3}+\varrho+\frac{1}{3}}\right\}. 
    \end{equation}
    for all $\mu>1/2$, $p\neq2$. If $p=2$ then the bound gains an extra factor of $\log\mu$. 
\end{lemma}

\section{Resolvent estimates}

 \subsection{Naive bound}
 In the following, we use the abbreviation $R_{0}(z):=(-\Delta-z)^{-1}$, $z\in\C\setminus[0,\infty)$.
 To prove the resolvent estimate~\eqref{eq. resolvent estimate}, we can assume $|z|=1$ (by scaling). We only consider the case $\re z>0$, since the case $\re z\leq 0$ is much easier.
 We split $R_0(z)=R_0(z)^{\rm high}+R_0(z)^{\rm low}$, where $R_0(z)^{\rm low}$ has Fourier multiplier $(\xi^2-\re z)^{-1}\chi(\xi/\sqrt{\re z})$ and $\chi$ is a smooth, radial function on $\R^d$ that equals $1$ on $B(0,2)$ and is supported on $B(0,4)$. Again, we only consider $R_0(z)^{\rm low}$ since the estimate for $R_0(z)^{\rm high}$ is much easier. Observing that 
 \begin{align}\label{spectral theorem for R0(z)}
  W_1R_0(z)^{\rm low}W_2=\int_0^{\infty}\frac{\chi(\lambda/\sqrt{\re z})}{\lambda^2-z}W_1\frac{\rd E(\lambda)}{\rd \lambda}   W_2 \,\rd \lambda
 \end{align}
and using the spectral measure bound \eqref{eq. spectral measure bound}, we find 
\begin{align}
  \|W_1R_0(z)^{\rm low}W_2\|_{\mathfrak{S}^p(L^2(\R^d))}\lesssim |\log|\im z||\|W_1\|_{L^{2q}}\|W_2\|_{L^{2q}}.   
\end{align}
This is almost the desired bound, up to a logarithm. At this point we note that, due to the spectral localization $\chi$, we may assume that $W$ is constant on the unit scale. Hence, in the following, we will assume that $W_1,W_2$ are supported on a union of $c$-cubes where $c\ll 1$.

\subsection{Smoothing}
The first step is to trade the logarithmic loss in $|\im z|$ into a logarithmic loss involving the size of the supports of $W_1, W_2$. 
Hence, we will temporarily assume $\supp(W_1), \supp(W_2)\subset B_R$, where $B_R$ is a ball in $\R^d$ with radius $R\gg 1$ (the center of the ball will not be important). The spatial localization to $B_R$ smooths off the integrand in \eqref{spectral theorem for R0(z)} on the $1/R$ scale, i.e., when using the triangle inequality, $|\lambda^2-z|$ may be replaced by $|\lambda^2-z|+1/R$. To make this rigorous, observe that, by the convolution theorem, 
\begin{align*}
\mathbf{1}_{B_R}m(D)\mathbf{1}_{B_R}=\mathbf{1}_{B_R}m_{R}(D)\mathbf{1}_{B_R}
\end{align*}
whenever $m(D)$ is a Fourier multiplier, where $m_R:=\varphi_R\ast m$, $\varphi_R(\xi):=R^d\varphi(R\xi)$ and $\varphi$ is a Schwartz function such that $\widehat{\varphi}=1$ on $B(0,2)$. 
The previous argument then yields
\begin{align}\label{R_0 log R loss}
  \|W_1R_0(z)^{\rm low}W_2\|_{\mathfrak{S}^p(L^2(\R^d))}\lesssim (\log R)|\|W_1\|_{L^{2q}}\|W_2\|_{L^{2q}}.   
\end{align}

\subsection{Removal of the logarithm}
To remove the $\log R$ factor, we partly follow the strategy of the first author and Merz \cite[Sect. 7.2]{cuenin2022random}. The idea
is reminiscent of an `epsilon removal
lemma' of Tao \cite{MR1666558} in the context of Fourier restriction theory.   
Compared to \cite{cuenin2022random}, the strategy is easier to explain in our setting since the potential is not random. Thus, we will only need to deal with the bilinear operator $(W_1,W_2)\mapsto W_1R_0(z)^{\rm low}W_2$, as opposed to multilinear operators with more than one factor of $R_0(z)^{\rm low}$, which had to be considered in \cite{cuenin2022random}. On the other hand, we work with Schatten norms, whereas \cite{cuenin2022random} deals with operator norms.

To illustrate the method, we will first prove a version without the radial symmetry assumption.

\begin{proposition}
\label{From spectral measure to resolvent bounds}
Let $d\geq 2$, $1\leq p_0,q_0<\infty$, and assume that the estimate
\begin{align}\label{prototype spectral measure estimate}
        \|W_1\mathcal{E}\mathcal{E}^*W_2\|_{\mathfrak{S}^{p_0}(L^2(\R^d))}\leq C(\log R)^s \|W_1\|_{L^{2q_0}}\|W_2\|_{L^{2q_0}}
    \end{align}  
holds for all $W_1,W_2\in L^{q}(\R^d)$ supported on a finite union of $c$-cubes contained in balls of radius $R$ (different balls are allowed) and for some $s>0$. Then, for $1\leq q<q_0$ and $1\leq p<p_0$ satisfying
\begin{align}\label{condition for p}
    \frac{1/p-1/p_0}{1/q-1/q_0}<\frac{1-1/p_0}{1-1/q_0},
\end{align}
there exists a constant $C_{p,q,s,d}$ such that 
\begin{align}\label{prototype resolvent estimate}
        \|W_1R_0(z)^{\rm low}W_2\|_{\mathfrak{S}^{p}(L^2(\R^d))}\leq C_{p,q,s,d}\|W_1\|_{L^{2q}}\|W_2\|_{L^{2q}}
    \end{align}      
for all $W_1,W_2\in L^{q}(\R^d)$ and for all $z\in\C\setminus[0,\infty)$ with $|z|=1$.    
\end{proposition}

\begin{remark}
    By \cite[Theorem 2]{MR3730931}, the assumed bound \eqref{prototype spectral measure estimate} is satisfied iff $q_0\leq (d+1)/2$ and $p_0\leq (d-1)q/(d-q)$. The resolvent bound \eqref{prototype resolvent estimate}, including the endpoint $q=q_0$, follows from \cite[Theorem 12]{MR3730931}. The Schatten exponent in both theorems there is $p=(d-1)q/(d-q)$. Using only \cite[Theorem 2]{MR3730931}, Proposition \ref{From spectral measure to resolvent bounds} asserts that \eqref{prototype resolvent estimate} holds for all $q<(d+1)/2$, with a slightly worse Schatten exponent $p$ strictly greater than, but arbitrarily close to, $(d-1)q/(d-q)$ (note that the latter is an increasing function of $q$).
    Thus, in practice, Proposition \ref{From spectral measure to resolvent bounds} does not give anything new. What is important is the technique, which will be used to prove Theorem \ref{theorem resolvent estimate}. 
\end{remark}

To free up notation, we will use $W, W'$ instead of $W_1, W_2$ in the proof of Proposition \ref{From spectral measure to resolvent bounds}. We will perform several decompositions of $W, W'$, and we start explaining these for a non-negative simple function $W$, supported on a finite union of $c$-cubes contained in a ball $B_R$. 
\begin{itemize}
    \item Horizontal dyadic decomposition: we write $W=\sum_{i\in\N}W_{i}$, where
    \begin{align}
        W_{i}\coloneqq W\mathbf{1}_{H_{i}\geq W\geq H_{i+1}},\hspace{0.5cm}H_{i}\coloneqq\inf\{\,t>0\,:\,|\{W>t\}|\leq2^{i-1}\,\},
    \end{align}
    then
\begin{align*}
\|H_i2^{i/q}\|_{\ell^r_i(\Z_+)}\asymp \|W\|_{L^{q,r}},
\end{align*}
where $L^{q,r}$ denotes a Lorentz space (see e.g.\ \cite[Thm. 6.6]{TaoNotes1247A}). Also note that $L^{q,q}=L^q$.\\
    \item Sparse decomposition: we write
\begin{align}\label{sparse decomp.}
W_i=\sum_{j=1}^{K_i}\sum_{k=1}^{N_i}W_{ijk},
\end{align}
where, for fixed $i,j$, the $W_{ijk}$ are supported on a ``sparse collection" of balls $\{B(x_k,R_i)\}_{k=1}^{N_i}$. By this we mean that the support of $W_{ijk}$ is contained in $B(x_k,R_i)$ and that the following definition is satisfied (cf.\ \cite[Def. 3.1]{MR1666558}) for some sufficiently large $\gamma$ (to be fixed later).
\end{itemize}
\begin{definition}
A collection $\{B(x_{k},R)\}_{k=1}^{N}$ is $\gamma$-sparse if the centers $x_{k}$ are $(RN)^{\gamma}$ separated.
\end{definition}
For fixed $\gamma>0$ and $K\geq 1$, \cite[Lemma 3.3]{MR1666558} asserts that \eqref{sparse decomp.} holds with 
\begin{align}\label{KiNiRi}
K_i=\mathcal{O}(K2^{i/K}),\quad N_i=\mathcal{O}(2^i),\quad R_i=\mathcal{O}(2^{i\gamma^K}).
\end{align}

We will also need the following estimate.
\begin{lemma}
    There exists a constant such that for all $z\in\C\setminus[0,\infty)$, $|z|=1$, for all $W,W'\in L^{q}(\R^d)$ supported on a finite union of $c$-cubes contained in balls of radius $R$ with $\dist(\supp(W),\supp(W'))>0$, and for all $\delta\in [0,1]$, we have 
    \begin{align}\label{decay bound resolvent almost trace norm}
        \|WR_0(z)^{\rm low}W'\|_{\mathfrak{S}^{1+\delta}}\leq C (\log R)\dist(\supp(W),\supp(W'))^{-\frac{(d-1)\delta}{1+\delta}}\|W\|_{L^{2}}\|W'\|_{L^{2}}.
    \end{align}
\end{lemma}

\begin{proof}
    The result will follow from interpolation between the two estimates
    \begin{align}\label{resolvent decay HS bound}
        \|WR_0(z)^{\rm low}W'\|_{\mathfrak{S}^2}\lesssim \dist(\supp(W),\supp(W'))^{-\frac{d-1}{2}}\|W\|_{L^{2}}\|W'\|_{L^{2}}
    \end{align}
and 
    \begin{align}\label{resolvent trace norm bound}
        \|WR_0(z)^{\rm low}W'\|_{\mathfrak{S}^2}\lesssim (\log R)\|W\|_{L^{2}}\|W'\|_{L^{2}}.
    \end{align}
The first estimate follows from the kernel bound
    \begin{align}
        |R_0(z)^{\rm low}(x-y)|\lesssim (1+|x-y|)^{-\frac{d-1}{2}},
    \end{align}
    which in turn is a consequence of a standard stationary phase argument (see e.g. \cite{MR3608659} or \cite{MR4153099}). The second estimate follows from 
\begin{align}\label{extension trace norm bound}
    \|W\mathcal{E}\mathcal{E}^*W'\|_{\mathfrak{S}^1}\lesssim \|W\|_{L^{2}}\|W'\|_{L^{2}}
\end{align}    
together with the smoothing argument leading to \eqref{R_0 log R loss}. Since $\mathbb{S}^{d-1}$ is compact, we have the obvious Hilbert-Schmidt bounds
\begin{align}
        \|W\mathcal{E}\|_{\mathfrak{S}^2}\lesssim \|W\|_{L^{2}},\quad \|\mathcal{E}^*W'\|_{\mathfrak{S}^1}\lesssim \|W'\|_{L^{2}},
\end{align}
which immediately imply \eqref{extension trace norm bound}. Interpolating \eqref{resolvent decay HS bound} and \eqref{resolvent trace norm bound} yields, for any $p\in [1,2]$,
   \begin{align}
        \|WR_0(z)^{\rm low}W'\|_{\mathfrak{S}^p}\lesssim (\log R)^{2/p-1}\dist(\supp(W),\supp(W'))^{-(d-1)(1-1/p)}\|W\|_{L^{2}}\|W'\|_{L^{2}}.
    \end{align}
Setting $p=1+\delta$ and recalling that $R\gg 1$ yields the claim.
\end{proof}

\begin{proof}[Proof of Proposition \ref{From spectral measure to resolvent bounds}]
We write $W,W'$ instead of $W_1,W_2$ and apply the above decomposition:
 \begin{align}
     W=\sum_{\alpha}W_{\alpha},\quad      W'=\sum_{\alpha'}W'_{\alpha'},
 \end{align}
 where
 $\alpha=(i,j,k)$, $\alpha'=(i',j',k')$ and $i,i'\in\Z_+$, $1\leq j\leq K_{i}$, $1\leq j'\leq K_{i'}$, $1\leq k\leq N_{i}$ and $1\leq k'\leq N_{i'}$. Set
 \begin{align*}
L_{\alpha,\alpha'}&:=1+\dist(\supp(W_{\alpha}),\supp(W'_{\alpha'})),\\
\rho_{\alpha,\alpha'}&:=2(1+L_{\alpha,\alpha'}+R_{i}+R_{i'}).
 \end{align*}
 In the following, we simply write $R_0$ instead of $R_0(z)^{\rm low}$.  
 By \eqref{R_0 log R loss},
 \begin{align}
 s_n(W_{\alpha}R_0W'_{\alpha'})\lesssim n^{-\frac{1}{p_0}} (\log \rho_{\alpha,\alpha'})^{s+1}\|W_{\alpha}\|_{L^{2q_0}}\|W'_{\alpha'}\|_{L^{2q_0}},
 \end{align}
 whereas \eqref{decay bound resolvent almost trace norm} implies
  \begin{align}\label{Euclidean sn interpolated}
 s_n(W_{\alpha}R_0W'_{\alpha'})\lesssim n^{-\frac{1}{1+\delta}} L_{\alpha,\alpha'}^{-\frac{(d-1)\delta}{1+\delta}}\|W_{\alpha}\|_{L^{2}}\|W'_{\alpha'}\|_{L^{2}}.
 \end{align}
 Interpolating the last two displays yields, for any $\theta\in (0,1)$,
 \begin{align}
 s_n(W_{\alpha}R_0W'_{\alpha'})\lesssim_{\theta} n^{-\frac{\theta}{1+\delta}-\frac{1-\theta}{p_0}}  L_{\alpha,\alpha'}^{-\frac{\theta(d-1)\delta}{2(1+\delta)}}\log^{s+1}(2+R_i+R_{i'})H_iH_{i'}'2^{(i+i')(\frac{\theta}{2}+\frac{1-\theta}{2q_0})}.   
 \end{align}
where we used the fact that
  \begin{align}
     L_{\alpha,\alpha'}^{-\frac{\theta(d-1)\delta}{1+\delta}}(\log^{s+1} \rho_{\alpha,\alpha'})^{1-\theta}\lesssim_{\theta}
     L_{\alpha,\alpha'}^{-\frac{(d-1)\delta}{2(1+\delta)}}\log^{s+1}(2+R_i+R_{i'}).
 \end{align}
 Summing \eqref{Euclidean sn interpolated} first over $k$, then $k'$, 
 \begin{align}
 \sum_{k,k'} s_n(W_{\alpha}R_0W'_{\alpha'})\lesssim_{\theta} n^{-\frac{\theta}{1+\delta}-\frac{1-\theta}{p_0}}  \log^{s+1}(2+R_i+R_{i'})H_iH_{i'}'2^{(i+i')(\frac{\theta}{2}+\frac{1-\theta}{2q_0})}.
 \end{align}
 Here we have used the fact that, for $\alpha=(i,j,k)$, $\alpha'=(i',j',k')$ and $i,j,i',j',k'$ fixed, the sum over $k$ is bounded,
\begin{align}\label{supsum}
\sum_{k\leq N_{i}} (1+\rd(B(x_{k},R_{i}),B(x_{k'},R_{i'}))^{-\frac{\theta(d-1)\delta}{2(1+\delta)}}\leq 1+2N_i(N_{i}R_{i})^{-\frac{\gamma\theta(d-1)\delta}{2(1+\delta)}}\lesssim 1,
\end{align}
uniformly in $i,j,i',j',k'$, provided $\frac{\gamma\theta(d-1)\delta}{2(1+\delta)}\geq 1$.
We will momentarily fix $\theta,\delta$, and then fix $\gamma$ such that $\frac{\gamma\theta(d-1)\delta}{2(1+\delta)}=1$.
Note that, even though the balls in \eqref{supsum} may belong to different sparse families, we have
\begin{align*}
\dist(B(x_{k},R_{i}),B(x_{k'},R_{i'}))\geq \frac{1}{2}(N_{i}R_{i})^{\gamma}
\end{align*}
for all but at most one $k$. Indeed, suppose for contradiction that this does not hold for two distinct $k_1,k_2$. Then, by the triangle inequality,
\begin{align*}
\rd(B(x_{k_1},R_{i}),B(x_{k_2},R_{i}))<(N_{i}R_{i})^{\gamma},
\end{align*}
which contradicts the sparsity of the collection $\{B(x_{k},R_{i})\}_{k=1}^{N_i}$.
 Summing over $j$ then $j'$ (recall $j\leq K_i$, $j'\leq K_{i'}$),
 \begin{align}
     \sum_{j,j'}\sum_{k,k'} s_n(W_{\alpha}R_0W'_{\alpha'})\lesssim_{\theta} K_{i}K_{i'} n^{-\frac{\theta}{1+\delta}-\frac{1-\theta}{p_0}}  \log^{s+1}(2+R_i+R_{i'})H_iH_{i'}'2^{(i+i')(\frac{\theta}{2}+\frac{1-\theta}{2q_0})}.   
 \end{align}
 Finally, summing over $i,i'$ and using \eqref{KiNiRi},
 \begin{align}
 \sum_{\alpha,\alpha'} s_n(W_{\alpha}R_0W'_{\alpha'})&\lesssim_{\theta} n^{-\frac{\theta}{1+\delta}-\frac{1-\theta}{p_0}}K^2\sum_{i,i'} \log^{s+1}(2+R_i+R_{i'})H_iH_{i'}'2^{(i+i')(\frac{\theta}{2}+\frac{1-\theta}{2q_0}+\frac{1}{K})}\\
 &\lesssim_{\theta,K,s}
 n^{-\frac{\theta}{1+\delta}-\frac{1-\theta}{p_0}}\sum_{i} H_i2^{i(\frac{\theta}{2}+\frac{1-\theta}{d+1}+\frac{2}{K})}\sum_{i'} H'_{i'}2^{i'(\frac{\theta}{2}+\frac{1-\theta}{2q_0}+\frac{2}{K})},
 \end{align}
 where we estimated $\log^{s+1} R_i\lesssim_{K,s} 2^{\frac{i}{K}}$ in the second line. We now fix $\theta,\delta>0$ and $K\geq 1$ in such a way that 
 \begin{align}
     &\frac{\theta}{2}+\frac{1-\theta}{2q_0}+\frac{2}{K}=\frac{1}{2q},\\
     &\frac{\theta}{1+\delta}+\frac{1-\theta}{p_0}<\frac{1}{p}.
 \end{align}
 This is always possible since $q<q_0$ and $p$ satisfies \eqref{condition for p}.
 Set $\eps:=\frac{1}{p}-\frac{\theta}{1+\delta}-\frac{1-\theta}{p_0}>0$. Then
 \begin{align}
     \sum_{\alpha,\alpha'} s_n(W_{\alpha}R_0W'_{\alpha'})&\lesssim_{q} n^{-\frac{1}{p}-\eps}\|W\|_{L^{2q,1}}\|W'\|_{L^{2q,1}},
 \end{align}
 and for any $N\in \N$, by Fubini,
 \begin{align}\label{Fubini}
    \sum_{\alpha,\alpha'} \sum_{n=1}^Ns_n(W_{\alpha}R_0W'_{\alpha'}) 
    \lesssim_{q} N^{1-\frac{1}{p}-\eps}\|W\|_{L^{2q,1}}\|W'\|_{L^{2q,1}}.
 \end{align}
Recall that, for a compact operator $\Sigma$,
\begin{align*}
\|\Sigma\|_{\mathfrak{S}_w^r}^{'}:=\sup_{n\in\N}n^{\frac{1}{r}}s_n(\Sigma)
\end{align*}
defines a quasinorm. For $2<r<\infty$,
an equivalent norm (hence satisfying the triangle inequality) is given by
\begin{align*}
\|\Sigma\|_{\mathfrak{S}_w^r}=\sup_{N\in\N} N^{-1+\frac{1}{r}}\sum_{n=1}^N s_n(\Sigma).
\end{align*}
Set $\frac{1}{r}:=\frac{1}{p}+\eps$. Then, by the triangle inequality and \eqref{Fubini},
\begin{align}
\|WR_0W'\|_{\mathfrak{S}_w^r}
\leq \sup_{N\in\N} N^{-1+\frac{1}{r}}\sum_{\alpha,\alpha'}\sum_{n=1}^Ns_n(W_{\alpha}R_0W'_{\alpha'})\lesssim_q \|W\|_{L^{2q,1}}\|W'\|_{L^{2q,1}}.
\end{align}
In particular, since $r<p$,
\begin{align}
   \|WR_0W'\|_{\mathfrak{S}^{p}}\lesssim_q \|W\|_{L^{2q}}\|W'\|_{L^{2q}},
\end{align}    
where we used that ${\mathfrak{S}_w^r}\subset \mathfrak{S}^{p_0}$ and $L^{2q,1}\subset L^{2q}$.
\end{proof}

\subsection{Proof of Theorem \ref{theorem resolvent estimate}}
The proof is similar to that of Proposition \ref{From spectral measure to resolvent bounds}. The difference is that we use \eqref{eq. spectral measure bound}, which only holds for radial functions, instead of \eqref{prototype spectral measure estimate} as input. Again, we write $W,W'$ instead of $W_1,W_2$ and decompose 
 \begin{align}
     w=\sum_{\alpha}w_{\alpha},\quad      w'=\sum_{\alpha'}w'_{\alpha'},
 \end{align}
 where $W(x)=w(|x|)$ and $W'(x)=w'(|x|)$. The balls in the sparse collections (in the proof of Proposition \ref{From spectral measure to resolvent bounds}) are now intervals, but \eqref{KiNiRi} still holds (this is just the one-dimensional case in \cite[Lemma 3.3]{MR1666558}). The remainder of the proof is the same as before.

\section{Eigenvalue sums}

\subsection{Proof of Theorem \ref{theorem sum of eigenvalues upper bound}} A special case ($\rho=0$) of \cite[Theorem 3.1]{MR3717979} states that if $p\geq 1$, $\sigma>0$ and if $K(z)$,
$z\in\C\setminus [0,\infty)$, is an analytic family of operators satisfying
\begin{align}\label{K(z) bound}
    \|K(z)\|_{\mathfrak{S}^p}\leq M|z|^{-\sigma}\quad\forall z\in\C\setminus [0,\infty),
\end{align}
then the finite-type eigenvalues $z_j$ of $1+K$ satisfy 
\begin{align}\label{Frank III special case}
    \sum_j \delta(z_j)|z_j|^{-\frac{1}{2}+\frac{1}{2}(2p\sigma-1+\eps)_+}\leq C_{p,\sigma,\eps} M^{\frac{1}{2\sigma}(1+(2p\sigma-1+\eps)_+)}
\end{align}
for any $\eps>0$. We use this with $K(z)=\sqrt{|V|}R_0(z)\sqrt{V}$, in which case the $z_j$ are the eigenvalues of $H=-\Delta+V$ (see \cite{MR3717979} for details). By Theorem \ref{theorem resolvent estimate}, since $V$ is assumed to be radial, \eqref{K(z) bound} holds with $M=C_{p,q,d}\|V\|_{L^{q}}$, $\sigma=1-\frac{d}{2q}$ and any $p>\frac{(d-1)q}{d-q}$. Since this implies that $p>\frac{q}{2q-d}$, we have $2p\sigma-1>0$, and hence we can omit the notation for the positive part, $(\ldots)_+$, in \eqref{Frank III special case}. Since we are imposing an open condition for $p$, we can absorb $\eps$ into $p$ and omit it from the inequality. The result is \eqref{eq. sums of eigenvalues main theorem}, so the proof of Theorem \ref{theorem sum of eigenvalues upper bound} is complete.

\subsection*{Acknowledgements} 
The first author was supported by the Engineering and Physical Sciences Research Council (EPSRC) [grant number EP/X011488/1]. The second author was supported by an EPSRC Doctoral Training Partnership (DTP) student grant at Loughborough University.

\bibliographystyle{plain}
\bibliography{Bibliography.bib}
\end{document}